\documentclass[12pt]{article}
\usepackage[english]{babel}
\usepackage{graphicx}
\usepackage{amsmath}
\usepackage{mathtools}
\usepackage{amsthm}
\usepackage{amssymb}
\usepackage{float}
\usepackage{geometry}
\usepackage{enumerate}
\usepackage{ragged2e}
\usepackage{array}
\usepackage{tikz}
\usepackage{vmargin}
\usepackage{marvosym}
\usepackage[minimal]{leadsheets}
\usepackage[new]{old-arrows}
\useleadsheetslibraries{musicsymbols}

\setmargins{2cm} 
{2cm}            
{17cm}           
{26cm}           
{0pt}            
{0cm}            
{0pt}            
{18pt}            

\newtheorem{thm}{Theorem}
\newtheorem{coro}{Corollary}
\newtheorem{lem}{Lemma}
\newtheorem{prop}{Proposition}

\theoremstyle{definition}
\newtheorem{defi}{Definition}

\newtheorem{rem}{Remark}

\title{On the Borel complexity of the space of left-orderings of nilpotent groups}
\date{}
\author{Emir Molina Taucán}

\begin{document}

\maketitle

\begin{center}
{\bf Abstract}
\end{center}
\noindent We give the first examples of nonabelian left-orderable groups such that the conjugacy orbit equivalence relation on its space of orders has infinity orbits, yet it is smooth in the Borel sense.  
The examples are all nilpotent groups and we provide a sufficient condition so that the space of orders of a nilpotent group has a smooth conjugacy orbit relation.
We also show with different examples that nilpotence is not a sufficient condition for smoothness.\\



\textbf{\centering Dedicated to the memory of Professors Marius M\u antiou and Manuel Pinto}

\section{Introduction}

Consider two groups $G$ and $H$ acting by homeomorphism on spaces $X$ and $Y$, respectively. Does there exist a criterion to compare both actions? Does there exist any form to say that one action is ``more complex than'' another? One answer to these questions is found in the area of Descriptive Set Theory 
thanks to the concepts of \textit{Borel reductions} and \textit{Borel complexity} \cite{Kanovei,Gao}.
To introduce these concepts, we need to restrict our attention to a particular kind of space.

Let $X$ be a set provided by a $\sigma$-algebra $\mathcal{B}$. The pair $(X,\mathcal{B})$ is called \textit{standard Borel space} if $\mathcal{B}$ is generated by a Polish topology $\mathcal{T}$. Recall that a Polish space is a separable, completely metrizable topological space. 
An equivalence relation $E$ on a standard Borel space $X$ is called Borel if $E\subseteq X\times X$ is Borel as a subset. 
We also say that $E$ is countable when every equivalence class is countable.
As an example of this concept, we could consider the relation induced by an action by continuous transformations $G\curvearrowright X$ of a countable group $G$ on a Polish space $X$.
In this case, we say that $X$ is a $G$-Polish space and denote this relation by $E_{G}^X$. 
The reader can find more interesting examples of Borel equivalence relations in Kechris' book \cite{NKechris}.



Consider now two Polish spaces $X$ and $Y$ with some Borel equivalence relations $E$ and $F$, respectively. We say that a Borel function $f:(X,E)\rightarrow (Y,F)$ is a {\it Borel reduction} if $xEy\Longleftrightarrow  f(x)F f(y)$. We then say that $E$ is Borel reducible to $F$, which is denoted by $E\leq_B F$, if there exists a reduction from $E$ to $F$. It is not hard to see that the relation $\leq_B$ induces a partial order in the class of standard Borel spaces provided of a Borel equivalence relation.
Therefore, it is possible to classify every equivalence relation in a ``hierarchy of grades'' of complexity. This concept is known as {\it Borel complexity}. 

The simplest Borel equivalence relations are called \textit{smooth}. We say that a relation $E$ is smooth if the quotient space $X/E$ is also a standard Borel space. Furthermore, the simplest nonsmooth relations are called \textit{hyperfinite}. By definition, a relation $E$ is hyperfinite if there exist finite equivalence relations $E_1\subseteq E_2 \subseteq \dots$ such that $E=\bigcup_{n=1}E_n$.
It is known that every smooth Borel equivalence relation is also hyperfinite \cite{NKechris}. Finally, the most complicated relations are called \textit{universal}. In this case, a relation is universal if every other countable equivalence relation is reduced to it. 
See \cite[Chapter 6]{NKechris} for more details about the different complexity grades.

The problem driving this work is to classify left-orderable groups with respect to their complexity grades. For completeness, we will give a little background on orderable groups.

A group $G$ is called left-orderable if it admits a linear order $\prec$ invariant by left multiplication. Such an order is called left-invariant. Moreover, we say that a left-invariant order is bi-invariant if it is also right-invariant. It must be remarked that every left-orderable group is torsion-free \cite[Proposition 1.3]{ClayRolfsen}.

An elementary result establishes that there exists a correspondence between left-invariant orders on a group and positive cones of it, where a {\it positive cone} $P$ of $G$ is a semigroup such that $G=P\ \amalg\ P^{-1}\ \amalg \{1\}.$ Here and throughout this work, $\amalg$ denotes a disjoint union. Note that a positive cone $P$ corresponds to a bi-order if and only if $g^{-1}Pg=P$ for all $g$ in $G$. 

Due to the previous correspondence, we can think of the space of all orders on $G$, denoted by $\mathcal{LO}(G)$, as a subspace of the space $\{0,1\}^{G}$. 
If $\{0,1\}^{G}$ is provided by prodiscrete topology, 
then $\mathcal{LO}(G)$ corresponds to a totally disconnected, compact, Hausdorff topological space \cite[Section 1.5]{ClayRolfsen}. 
Moreover, if $G$ is a countable group, then $\mathcal{LO}(G)$ is also a Polish space. 

From the topology on $\mathcal{LO}(G)$, we see that for every positive cone $P$ 
a basic neighborhood of $P$ in $\mathcal{LO}(G)$ is given $U=\{\ Q\in \mathcal{LO}(G)\ : \ t_1,\dots,t_p\in Q\ \}$, where the (finitely many) elements $t_1,\dots, t_p$ belongs to $P$.

For more information on left-orderable groups, the reader can see \cite{ClayRolfsen} and \cite{GOD}.

An important fact about any orderable group $G$ is that it acts by homeomorphisms on $\mathcal{LO}(G)$ by 
\begin{align*}
    \mathcal{LO}(G)\times G&\longrightarrow \mathcal{LO}(G),\\(P,g)&\longrightarrow g^{-1}Pg.
\end{align*} 
If $G$ is a countable group, the previous action allows us to think of $\mathcal{LO}(G)$ as a $G$-Polish space and think of its induced orbit relation, which we denote by $E_{lo}(G)$, as a countable Borel equivalence relation.

Originally, Deroin, Navas and Rivas proposed the following question in their book \textit{Groups, Orders and Dynamics} \cite{GOD}: {\it ``Can the space of orbits $\mathcal{LO}(\Gamma)/\Gamma$ be a nonstandard Borelian space for a left-orderable group $\Gamma$?''}.
This question has a positive answer and has been developed by Adam Clay and Filippo Calderoni through their three articles \cite{ClayCalderoni1}, \cite{ClayCalderoni2} {\color{violet} and} \cite{ClayCalderoni3}. Thanks to their work, it is known that for free groups $\mathbb{F}_n$, noncyclic knot groups, Baumslag-Solitar $BS(1,n)$ groups (where $n\geq 2$) and Thompson's group $F$  the space of orbits is, in fact, a nonstandard Borelian space. Moreover, they have proved that $E_{lo}(\mathbb{F}_n)$ is a universal countable Borel equivalence relation, for $n\geq 2$. However, at the time of this writing, the only known examples of groups with a smooth relation $E_{lo}(G)$ are abelian groups, or groups such that their space of order has only bounded-finite orbits.
Thus, the motivation of this work is to provide the first family of examples of left-orderable groups with infinite orbits in $\mathcal{LO}(G)$, yet where $E_{lo}(G)$ is a smooth equivalence relation. 

All the examples presented in this work are nilpotent groups. For it, we are interested in determining the Borel complexity grade of any nilpotent group. A recent result shown by Schneider and Seward asserts that any orbit equivalence relation induced by a Borel action of a locally nilpotent group on a standard Borel space is hyperfinite \cite[Theorem 1.1]{SS}. 
This implies that $E_{lo(G)}$ must be a hyperfinite relation when $G$ is nilpotent left-orderable group. Therefore, the problem is reduced to determining whether $E_{lo}(G)$ is smooth or not. Our first main result is a sufficient condition for smoothness on $E_{lo}(G)$ for finitely generated nilpotent groups $G$.

\begin{thm}
\label{thmsmooth}
    Let $G$ be a left-orderable finitely generated nilpotent group. If $[G,G]$ is cyclic and $[G,G]\subseteq Z(G)$, then the relation $E_{lo}(G)$ is smooth.
\end{thm}

A family of examples satisfying the hypothesis of Theorem \ref{thmsmooth} are the generalized discrete Heisenberg groups $\mathcal{H}_{2n+1}$, which are defined as the matrix group
$$\mathcal{H}_{2n+1}:=\left \{\ \left( \begin{array}{ccc}
1 & \vec{a} & c\\
\vec{0}^t & I_n & \vec{b}^t\\
0 & \vec{0} & 1  
\end{array} \right )\ :\ \vec{a},\vec{b}\in \mathbb{Z}^n,\ c\in\mathbb{Z}\ \right \}$$
for all $n\geq1.$ Thus, we get 

\begin{coro}
    The relation $E_{lo}(\mathcal{H}_{2n+1})$ is smooth, for all $n\geq1.$
\end{coro}

It should be noted that the space $\mathcal{LO}(\mathcal{H}_{2n+1})$ contains orders with infinite conjugacy {\color{violet}orbits} (see Section 2.3). 

Given Theorem \ref{thmsmooth}, one could guess that $E_{lo}(G)$ is smooth for all left-orderable nilpotent groups. However, this is not true{\color{violet},} as shown by our second main result.

\begin{thm}
\label{nosmooth}
    There exists a nilpotent left-orderable group $G$ such that $E_{lo}(G)$ is not smooth.
\end{thm}

We prove Theorem \ref{nosmooth} in Section 3.2 and we verify that the nilpotent group of $k$-by-$k$ lower triangular matrices with integer entries
and 1 in the diagonal satisfy Theorem \ref{nosmooth} for $k\geq 4$ (see Section 3.3). In this way, we present a countable family of nilpotent groups whose conjugacy orbit relation on its space of orders is not smooth.\\  

\noindent\textbf{Acknowledgment:} We are grateful to Martín Gilabert for useful discussions on the subject and to Andrés Navas, Adam Clay and Filippo Calderoni for valuable comments. The author extends special thanks to Cristóbal Rivas for helping to develop this work.
The author was partially funded by the Fondecyt Project 1241135 and the ECOS 23003 project:
``Small spaces under action''.
\section{Preliminaries}

\subsection{Smoothness in $G$-Polish spaces}

As we have mentioned in the Introduction, we are interested in verifying if $E_{lo}(G)$ is or is not a smooth equivalence relation when $G$ is a nilpotent group. Recall this concept. 

A Borel equivalence relation $E$ on a standard Borel space $X$ is called {\it smooth} if the quotient space $X/E$ is also a standard Borel space. Equivalently, {\it $E$ is a smooth equivalence relation if and only if there exists a standard Borel space $Y$ equipped with the identity relation $\Delta_Y$, such that $E\leq_B\Delta_Y$} \cite{ClayCalderoni1}.

At this point, we do not have any useful tool to verify whether a Borel equivalence relation is smooth or not. Fortunately, Osin shows a dynamic equivalence for smoothness when we consider a group action by homeomorphisms on a Polish space.

\begin{prop}\cite[Proposition 2.7]{Osin}
\label{condpoint}
    {\it Suppose that $G$ is a countable group and $X$ is a Polish $G$-space. Then the following are equivalent:
\begin{itemize}
    \item[(1)] $E_{G}^X$ is smooth.
    \item[(2)] There are no condensed points in $X$.
\end{itemize}}
\end{prop}
Here, Osin calls {\it condensed point} to an element that it is a limit point of its own orbit. Therefore, the previous proposition lets us conclude that smoothness is equivalent to the fact that every orbit in a $G$-space $X$ is discrete as a subspace.\\

\noindent\textit{Example.} Consider a rational rotation action of $\mathbb{Z}$ on the circle $S^1$. Given that every orbit in $S^1$ is finite, there are no condensed points in $S^1$. Then the relation $E_{\mathbb{Z}}^{S^1}$ is smooth.
On the other hand, if we consider a irrational rotation action, then every point in $S^1$ is a condensed point. Then, in this case, $E_{\mathbb{Z}}^{S^1}$ is not a smooth equivalence relation.

\subsection{On orders on nilpotent groups}

In this section, 
we give background on orders on finitely generated nilpotent groups.
It is known that \textit{every left-invariant order on a nilpotent group is Conradian} thanks to a result proved by J.C. Ault \cite{Ault}. See also a dynamical approach from \cite{Navas2} and \cite{NaRivas}.
By definition, a Conradian order on a group $G$ is an left-invariant order $\prec$ that satisfies that for all $id\prec f$ and $1\prec g$ we have $g\prec fg^n$ for some $n\in \mathbb{N}$. Note that every bi-order is a Conradian order, although the converse is not true. However, the Conradian orders, as the bi-orders, have important structural consequences on the group, that we now present.

Consider an order $\prec$ in $G$. We say that a subset $C\subseteq G$ is {\it convex} relative to $\prec$ if for all $g,h\in C$ and $f\in G$ the implication $g\prec f\prec h\implies f\in C$ holds. We define the notion of \textit{convex subgroup} analogously. Moreover, we say that a pair $(C,D)$ is a {\it convex jump} (with respect to $\prec$) if $C\subset D$, both $C$ and $D$ are convex subgroups relative to $\prec$ and there is no strict convex subgroup between them. The next result is known as Conrad's Theorem \cite{Conrad} (see also \cite{Botto}):

\begin{thm}\textbf{(Conrad)} {\it A left-order $\preceq$ on a group $G$ is said to be Conradian if the following four equivalent properties hold:
\begin{itemize}
    \item[(1)] For all $f\succ id $ and $g\succ id$ (for all positive $f,g$, for short), we have $fg^n\succ g$ for some $n\in \mathbb{N}$.
    \item[(2)] If $id\prec g\prec f$, then $g^{-1}f^ng\succ f$ for some $n\in \mathbb{N}$.
    \item[(3)] For all positive $g\in G$, the set $S_g=\{f\in G\ |\ f^n\prec g,\ \textit{for all } n\in\mathbb{Z}\}$ is a convex subgroup.
    \item[(4)] For every $g$, we have that $G_g$ is normal $G^g$, and there exists a nondecreasing group homomorphism (to be referred to as the Conrad homomorphism) $\tau_{\preceq}^g:G^g\rightarrow \mathbb{R}$ whose kernel coincides with $G_g$. Moreover, this homomorphism is unique up to multiplication by a positive real number.
\end{itemize}
}
\end{thm} 
In this statement, $G_g$ (resp. $G^g$) denotes the largest (resp. smallest) convex subgroup that does not contain $g$ (res. contains $g$).
The reader can note that if $(C,D)$ is a convex jump with respect to a Conradian order, then $C=G_g$ and $D=G^g$ for some $g\in G$.
Finally, we conclude from this Theorem that a left-invariant order $\prec$ in $G$ is Conradian if every convex jump $(C,D)$ satisfies $C\vartriangleleft D$ and there exists a nondecreasing morfisphism $\tau:D\rightarrow \mathbb{R}$ whose kernel is $C$. 

Now, let $G$ be a finitely generated left-orderable group. Given an order $\prec$ in $G$, there exists a unique maximal proper convex subgroup relative to $\prec$. If $G$ is also a nilpotent group, then every subgroup of $G$ is also finitely generated \cite[Chapter 5]{Robinson}. Hence, from an inductive argument, there exists a unique finite subgroup series
$$\{1\}=C_{n+1}\leq C_{n} \leq \dots\leq  C_0=G$$
such that $C_{i+1}$ is the maximal proper convex subgroup relative to order $\prec$ restricted to $C_i$. 
We conclude the following:

\begin{prop}
\label{series}
Let $G$ be a finitely generated left-orderable group. If $G$ is nilpotent, then every order in $G$ is Conradian. Moreover, if $\prec$ is an order on $G$, there exists a unique subnormal series
$$\{1\}=C_{n+1}\vartriangleleft C_{n} \vartriangleleft \dots \vartriangleleft  C_0=G$$
and nondecreasing morphisms $f_i:C_i\rightarrow \mathbb{R}$ which Ker$f_i=C_{i+1}$.
\end{prop}

In contrast, if we consider an short exact sequence
$$1\longrightarrow K\longhookrightarrow  G\overset{p}\longrightarrow H\longrightarrow 1$$
where $(K,\prec_K)$ and $(H,\prec_{H})$ are left-ordered groups, then $G$ admits an order in a lexicographic way; this means, $g\prec g_0$ if and only if $p(g)\prec_Hp(g_0)$ or $p(g)=p(g_0),$ but $1_G\prec_Kg^{-1}g_0$ \cite[Problem 1.8]{ClayRolfsen}. From this fact it follows that every subnormal series, whose factors are left-orderable groups, induces a lexicographic order on $G$. 
In particular, for every (finitely generated) nilpotent group $G$ and every positive cone $P$ of it, we know that $C_{i}/C_{i+1}$ corresponds to a finitely generated abelian group for all $i=0,1,\dots, n-1$. Therefore, $P$ is defined lexicographically by orders in $\mathbb{Z}^d$'s.

Let $\prec$ be an order on a finitely generated nilpotent group $G$ and $P$ and $\{1\}=C_{n+1}\vartriangleleft C_n\vartriangleleft \dots \vartriangleleft C_0=G$ its associated positive cone and subnormal series, respectively. Since $C_1$ is a normal subgroup of $G$, we know that $g^{-1}(P\cap C_1)g\in \mathcal{LO}(C_1)$ for all $g\in G$. In fact, the previous property holds for any $C_i$ that is normal in $G$.
This motivates the following definition.

\begin{defi}
Let $P$ be a positive cone of a finitely generated nilpotent group $G$ and $\{1\}=C_{n+1}\vartriangleleft C_n\vartriangleleft \dots \vartriangleleft C_0=G$ its associated (sub)normal series of convex subgroups. For all $0\leq i\leq n$, we call to $P_i:=P\cap (C_i\setminus C_{i+1})$ the component of $P$ associated with subgroup $C_i$.
\end{defi}

Let us suppose that $C_i$ and $C_{i+1}$ are both normal in $G$ for some index $i$. It is obvious that $g^{-1}(C_i\setminus C_{i+1})g\subseteq C_i\setminus C_{i+1}$ for all $g\in G$. 
In this case, we say then conjugation action of $G$ preserves the component $P_i$, since $g^{-1}P_ig=g^{-1}(P\cap C_i\setminus C_{i+1})g\subset C_i\setminus C_{i+1}.$
Therefore, if the conjugation action preserves the first $m$ components of $P$, we have that
\begin{equation}
\label{equation1}
g^{-1}Pg=g^{-1}P_0g\ \amalg\ \dots\ \amalg g^{-1}P_{m-1}g\ \amalg\ g^{-1}(P\cap C_m)g  
\end{equation}
for all $g\in G$.

Let us suppose now that $G$ not only preserves some component $P_i$ of $P$, but $G$ acts trivially on $C_i/C_{i+1}$. This implies that $g^{-1}P_ig=P_i$ for all $g\in G$. In this case, we say that $G$ fixes the component $P_i$.
In particular, if $G$ preserves and fixes the first $m$ components of $P$, we conclude from equation (\ref{equation1}) that
\begin{equation}
\label{equation2}
g^{-1}Pg=P\cap (G\setminus C_m)\ \amalg\ g^{-1}(P\cap C_m)g  
\end{equation}
for all $g\in G$. 

Therefore, when equation (\ref{equation2}) holds, we only need to study the conjugacy orbit of $P\cap C_m$  to study the conjugacy orbit of $P$. For it, we can use the fact that the automorphism group Aut$(G)$ acts by homeomorphisms on $\mathcal{LO}(G)$ for every left-orderable group $G$ in the natural way; this means $\mathcal{LO}(G)\times \text{Aut}(G)\rightarrow \mathcal{LO}(G)$, where $(P,\varphi)\rightarrow \varphi^{-1}(P)$ \cite[Chapter 10]{ClayRolfsen}. 
We develop this idea in the following:

\begin{lem}
\label{dico}
Let $G$ be a finitely generated nilpotent left-orderable group that preserves and fixes the first $m$ components of some positive cone $P$ and denote $P^*:=P\cap C_m$. Then:
\begin{itemize}
    \item[i)] If Orb$_G(P^*)$ is a discrete subspace of $\mathcal{LO}(C_m)$, then $P$ is not condensed in $\mathcal{LO}(G)$. 
    \item[ii)] If $P^*$ is a limit point of Orb$_G(P^*)$, then $P$ is condensed in $\mathcal{LO}(G)$.
\end{itemize}
\end{lem}

    \indent\textit{Proof of i)} If Orb$_G(P^*)$ is discrete in $\mathcal{LO}(C_m)$, then there exist elements $b_1,\dots, b_s\in C_m$ and a neighborhood $U^*=\{\ Q^*\in \mathcal{LO}(C_m)\ |\ b_1,\dots, b_s\in Q^*\ \}$ of $P^*$ such that Orb$_G(P^*)\cap U^*=\{P^*\}$.
    Thanks to equation (\ref{equation2}), the reader can conclude that Orb$_G(P)\cap U=\{P\}$
    where $U=\{\ Q\in \mathcal{LO}(G)\ |\ b_1,\dots, b_s\in Q\ \}$. Therefore, $P$ is not condensed in $\mathcal{LO}(G).$
    \\

    \textit{Proof of ii)} Let $U$ be an arbitrary neighborhood of $P$ in $\mathcal{LO}(G)$. By definition, there exist elements $a_1,\dots, a_t\in G\setminus C_m$ and $b_1,\dots, b_s\in C_m$ such that
    $$U=\{\ Q\in \mathcal{LO}(G)\ |\ a_1,\dots, a_t,b_1,\dots, b_s\in Q\ \}.$$

    Consider then the neighborhood $U^*=\{\ Q^*\in\mathcal{LO}(C_m)\ |\ b_1,\dots, b_s\in C_m\ \}$ of $P_m$ in $\mathcal{LO}(C_m)$. 
    By hypothesis, there exists an element $g\in G$ such that $g^{-1}P^*g\in U^*\setminus\{P^*\}$. In particular, $b_1,\dots, b_s\in g^{-1}P^*g$.     
    Moreover, from equation (\ref{equation2}), it follows that $a_1,\dots, a_t\in P\cap (G\setminus C_m)=g^{-1}[P\cap(G\setminus C_m) ]g$. Then $g^{-1}Pg\in U\setminus\{P\}$. Finally, $P$ is a condensed point in $\mathcal{LO}(G)$.$\hfill\square$

\subsection{On orders on abelian groups}

\indent From here until the end of this article, we denote by $B_0=\{\vec{e}_1,\dots, \vec{e}_n\}$ the canonical basis of $\mathbb{Z}^n\subseteq \mathbb{R}^n$ and $\langle\ ,\ \rangle$ the usual inner product of $\mathbb{R}^n$.

Consider a free abelian group $\mathbb{Z}^n$, for $n\geq1$. 
Recall that every order on a nilpotent group must be Conradian. In particular, by Proposition \ref{series}, for every order $\prec$ on $\mathbb{Z}^n$ there exists a finite (sub)normal series $$\{1\}=C_{m+1}\vartriangleleft C_m\vartriangleleft\dots \vartriangleleft  C_0=\mathbb{Z}^n$$
and nondecreasing homomorphisms $f_i:C_i\subseteq \mathbb{Z}^n\rightarrow \mathbb{R}$ such that Ker$f_i=C_{i+1}$.

Let $P$ the associated positive cone and consider $f_0$ the first homomorphism from above. If we let the vector $\vec{v}_0=(f_0(\vec{e_1}),\dots, f_0(\vec{e}_n))$, then for all $\vec{w}=b_1\cdot\vec{e}_1+\dots+ b_n\cdot \vec{e}_n\in \mathbb{Z}^n$ we have that
$$f_0(\vec{w})=b_1f_0(\vec{e}_1)+\dots + b_nf_0(\vec{e}_n)=\langle\vec{v_0},\vec{w}\rangle.$$

It follows from this expression that $\vec{w}\in P\cap (\mathbb{Z}^n\setminus C_1)$ if and only if $\langle\vec{v}_0,\vec{w}\rangle>0$.
More generally, for all $i=0,\dots, m$ there exists a vector $\vec{v}_i\in \mathbb{R}^n$ such that $f_i=\langle\vec{v}_i, \cdot\rangle$ and, therefore, $\vec{w}\in P_i=P\cap (C_i\setminus C_{i+1})$ if and only if $\vec{w}\in C_i$ and $f_i(\vec{w})=\langle \vec{v}_i,\vec{w}\rangle>0.$
We conclude that a vector $\vec{w}\in \mathbb{Z}^n$ is in $P$ if and only if there exists $i_0\leq m$ such that $\langle \vec{v}_{i_0},\vec{w}\rangle>0$ and $\langle\vec{v}_i,\vec{w}\rangle=0$ for all $i<i_0$. 
In this case, we say that $P$ is defined by the vectors $\vec{v}_0,\vec{v_1},\dots, \vec{v}_m$.

The following correspond to a simple but useful result:

\begin{prop}
\label{vectors}
Let $P$ be a positive cone of $\mathbb{Z}^n$ defined by $\vec{v}_0,\dots,\vec{v}_t\in \mathbb{R}^n$, and $A$ a matrix in $GL_{n}(\mathbb{Z})$. Then the order $A(P)$ is defined by $\vec{u}_0,\dots,\vec{u}_t\in\mathbb{R}^n$, where $\vec{u_i}:=(A^{-1})^t(\vec{v_i})$ for all indices $i$.
\end{prop}

\begin{proof}
Thanks to $\langle\vec{w}, A(\vec{u})\rangle=\langle A^{t}(\vec{w}), \vec{u}\rangle$ for all $\vec{w},\vec{u}\in \mathbb{R}^n$, we have that
\begin{align*}
    \vec{m}\in A(P)\ &\Longleftrightarrow\ A^{-1}(\vec{m})\in P\\
    &\Longleftrightarrow\ \langle\vec{v_{i_{0}}}, A^{-1}(\vec{m})\rangle>0 \text{ for some } i_0, \text{ and }\ \langle\vec{v_i}, A^{-1}(\vec{m})\rangle=0,\ \text{for all } i<i_0\\
    &\Longleftrightarrow\ \langle(A^{-1})^t(\vec{v_{i_{0}}}),\vec{m}\rangle>0 \text{ for some } i_0, \text{ and }\ \langle(A^{-1})^t(\vec{v_i}), \vec{m}\rangle=0,\ \text{for all } i<i_0. 
\end{align*}
This shows what we want.   
\end{proof}

Of course, the conjugacy action of a group $G$ restricted to a normal subgroup $H$ admits a matrix representation when $H$ is isomorphic to a finitely generated free abelian group. For it, Proposition \ref{vectors} let us determine the conjugacy orbit of every order on $H$ in that case. 
For our purposes, we are interested in the action of a particular subgroup of $SL_n(\mathbb{Z})$ which we will denote by $L_n$ and define as
$$L_n:=\left \{\ A=\left (
\begin{array}{cc}
    I_{n-1} & \vec{0}^t\\
    \vec{a} & 1
\end{array}
\right )\  :\ \vec{a}=(a_1,\dots,a_{n-1})\in \mathbb{Z}^{n-1} 
\right \}.$$

\begin{rem}
Let $n$ be a positive integer and $\mathcal{H}_{2n+1}$ be the $(2n+1)-$dimensional discrete Heisenberg group defined as in the Introduction. Let us consider the elements
$$x_i:=\left ( 
\begin{array}{ccc}
1 & \vec{e_i} & 0\\
\vec{0}^t & I_n & \vec{0}^t\\
0 & \vec{0} & 1
\end{array}
\right ), \
y_i:=\left ( 
\begin{array}{ccc}
1 & \vec{0} & 0\\
\vec{0}^t & I_n & \vec{e_i}^t\\
0 & \vec{0} & 1
\end{array}
\right ) \text{ and }
z:=
\left ( 
\begin{array}{ccc}
1 & \vec{0} & 1\\
\vec{0}^t & I_n & \vec{0}^t\\
0 & \vec{0} & 1
\end{array}
\right )
$$
for all $i\in \{1,2,\dots ,n\}$. 

The reader can check that every element in $\mathcal{H}_{2n+1}$ is expressed as $y_1^{b_1}\dots y_n^{b_n}x_1^{a_1}\dots x_n^{a_n} z^c$, for some integers $a_i$'s, $b_j$'s and $c$. 
Moreover, the elements of this generating set satisfy the relations $[x_i,y_i]=z$ for all $i$, while any other commutator is trivial. This preferred normal form for $\mathcal{H}_{2n+1}$ allows us to conclude that $\mathcal{H}_{2n+1}$ is isomorphic to $\langle\ y_1,\dots y_n\ \rangle \ltimes \langle\ x_1,\dots,x_n,z\ \rangle $, where $\langle\ y_1,\dots y_n\ \rangle\simeq \mathbb{Z}^n$ and $\langle\ x_1,\dots,x_n,z\ \rangle\simeq \mathbb{Z}^{n+1}$. We then conclude that $\mathcal{H}_{2n+1}$ admits the short exact sequence 
$$1\longrightarrow \mathbb{Z}^{n+1}\longhookrightarrow  \mathcal{H}_{2n+1}\overset{p}\longrightarrow \mathbb{Z}^n\longrightarrow 1.$$
and, therefore, $\mathcal{H}_{2n+1}$ is a left-orderable group. Moreover, this group admits infinitely many orders defined lexicographically by orders on $\mathbb{Z}^{n+1}$ and $\mathbb{Z}^n$, respectively.  

Consider now a positive cone $Q$ of $\langle\ x_1,\dots,x_n,z\ \rangle\simeq \mathbb{Z}^{n+1}$. To avoid confusion, we denote by $D=\{\vec{d}_1,\dots,\vec{d}_n,\vec{d}_{n+1}\}$ the canonical basis of $\mathbb{Z}^{n+1}$ and let the assignation given by $x_i\rightarrow \vec{d}_i$ (for $1\leq i\leq n$) and $z\rightarrow \vec{d}_{n+1}$. Given the relations between the elements of our preferred generating set, the restricted conjugated action of $\mathcal{H}_{2n+1}$ on $\langle\ x_1,\dots,x_n,z\ \rangle$ is represented by the homomorphism $\varphi:\mathcal{H}_{2n+1}\rightarrow SL_{n+1}(\mathbb{Z})$, where $\varphi(y_i)=\left (
\begin{array}{cc}
    I_{n} & \vec{0}^t \\
    \vec{e}_i & 1
\end{array}
\right )\in L_{n+1}$ for all $i=1,\dots, n$. Meanwhile, every other generating element is map to the identity matrix.

It follows from Proposition \ref{vectors} that if $Q$ is defined by vectors $\vec{v}_0,\dots,\vec{v}_m\in\mathbb{R}^{n+1}$, then $\varphi(h)(Q)$ is defined by $\varphi(h^{-1})^t(\vec{v}_0),\dots, \varphi(h^{-1})^t(\vec{v}_m)$, for all $h\in \mathcal{H}_{2n+1}$. In particular, if a positive cone $Q_0$ is defined uniquely by a vector $\vec{v}_0=(\alpha_1,\dots,\alpha_{n},\alpha_{n+1})\in \mathbb{R}^{n+1}$, where $\alpha_{n+1}\neq 0$, then every element $h$ in $\langle\ y_1,\dots y_n\ \rangle$ defines a different positive cone $\varphi(h)(Q_0)$. It is sufficient to note that $\varphi(h^{-1})^t(\vec{v}_0)\neq \varphi(k^{-1})^t(\vec{v}_0)$ if $h$ and $k$ are different elements of $\langle\ y_1,\dotsm ,y_n\ \rangle$. A posteriori, orders such as $Q_0$ let us define a lexicographic order on $\mathcal{LO}(\mathcal{H}_{2n+1})$ whose conjugacy orbit is infinite.
\end{rem}

\section{Main Results}

To prove Theorem \ref{thmsmooth}, we establish the following result with respect to the action of $L_n$, the group defined in the previous section, on the space $\mathcal{LO}(\mathbb{Z}^n)$.

\begin{prop}
\label{matrixact}
    Let $n$ be a positive integer. The canonical action $L_n\curvearrowright \mathcal{LO}(\mathbb{Z}^n)$ has no condensed points. In particular, the induced equivalence relation is smooth.
\end{prop}

\begin{rem}
In \cite{Poulin}, A. Poulin considers the more general case of $GL_n(\mathbb{Z})$ acting on $\mathcal{LO}(\mathbb{Z}^n)$, showing that this action is never smooth for $n\geq 2$ and, in fact, not hyperfinite for $n\geq 3$.   
\end{rem}

Before we prove Proposition \ref{matrixact}, we need to develop a previous lemma.

\begin{lem}
\label{lemacono}
    Let $\vec{v}=(x_1,\dots,x_n)$ be a vector in $\mathbb{R}^n$, where $x_n>0$, and $A$ a matrix in $L_n\setminus\{I_n\}$.
    Then there exists an integer vector $\vec{w}\in \mathbb{Z}^n$ such that $\langle \vec{v} , \vec{w} \rangle>0$ but $\langle (A^{-1})^t(\vec{v}) , \vec{w} \rangle<0.$
\end{lem}
\begin{proof}
Let us denote by $H_{v}$ the orthogonal complement of $\langle\ \vec{v}\ \rangle$ and define $\vec{w_i}:= \text{sgn}(a_i)\cdot(\vec{e_i}-\frac{x_i}{x_n}\vec{e}_n)$ for all $i=1,\dots,n-1$. 
Here, we use the convention sgn$(0)=1$. 

The reader can verify that $\{\vec{w}_1,\dots, \vec{w}_{n-1}\}$ is a $\mathbb{R}$-linearly independent set and, therefore, $H_v=\langle\ \vec{w}_1,\dots, \vec{w}_{n-1}\ \rangle$.
Moreover, given that $\langle\vec{v},\vec{x}\rangle=\langle A^t(A^{-1})^t(\vec{v}),\vec{x}\rangle =\langle (A^{-1})^t(\vec{v}),A(\vec{x})\rangle$ for all $\vec{x}\in \mathbb{R}^n$, we also conclude that $H_{(A^{-1})^tv}=\langle A(\vec{w}_1),\dots, A(\vec{w}_{n-1})\rangle.$

A simple calculus shows that $A(\vec{w}_i)=\vec{w}_i+|a_i|\cdot \vec{e_n}$. Thus $\langle \vec{v},A(\vec{w}_i)\rangle=\langle\vec{v},\vec{w}_i\rangle +|a_i|\langle \vec{v},\vec{e_n}\rangle=|a_i|x_n\geq 0$ for all indices $i$.
Of course, we have that $\langle\vec{v},A(\vec{w}_i)\rangle>0$ when $a_i\neq 0.$

Consider now $i_0$ the smallest index $i$ such that $a_i\neq 0$ and fix the vector $A(\vec{w}_{i_0})=\vec{w}_{i_0}+|a_{i_0}|\cdot\vec{e}_n$. We also consider the convex cone $C$ generated by $\vec{w}_1,\dots, \vec{w}_{i_0},A(\vec{w}_{i_0}),\dots, \vec{w}_{n-1}$. Recall that $C=\{\ \sum_{i=1}^{n-1}\alpha_i\cdot\vec{w}_i  + \beta_{i_0}A(\vec{w}_{i_0})\ |\ \alpha_{i},\beta_{i_0}\in \mathbb{R}_{\geq 0}\ \}$.

 To conclude with our proof, we need the following:\\ 

    \noindent\textit{Claim.} If $B=\{\vec{v}_1,\dots, \vec{v}_n\}$ is a basis of the vector space $\mathbb{R}^n$ and $C$ the convex cone generated by the same elements, then Int$(C)\cap \mathbb{Z}^n\neq \emptyset$.\\

    \noindent\textit{Proof of Claim.} Let us consider the convex cone $C_0$ generated by the canonical basis $B_0$. We also consider the change-of-basis matrix $A$ of $B_0$ to $B$. The reader can note that $A: \mathbb{R}^n\rightarrow \mathbb{R}^n$ is a homeomorphism such that $A(C_0)=C$. 
    
    Due to Int$(C_0)$ being trivially nonempty and $A$ being a homeomorphism, we conclude that Int$(C)$ is also nonempty. Then there exist $\vec{w}\in C$ and $\epsilon>0$ such that $B(\vec{w},\epsilon)\subseteq C$, where $B(\vec{w},\epsilon)$ denoted the open ball with radius $\epsilon$ (with respect to the Euclidean metric in $\mathbb{R}^n$) and center $\vec{w}$. Moreover, $B(\lambda\vec{w},\lambda\epsilon)=\lambda B(\vec{w},\epsilon)\subseteq C$ for all $\lambda>0$.
    
    Finally, there must exist $\lambda>0$ large enough such that $B(\lambda\vec{w},\lambda\epsilon)\cap \mathbb{Z}^n\neq \emptyset$. $\hfill\square$\\

A quick thought will convince us that $\{\vec{w}_1,\dots, \vec{w}_{i_0},A(\vec{w}_{i_0}),\dots, \vec{w}_{n-1}\}$ is a basis of $\mathbb{R}^n$. From the previous claim, it follows that there exists a integer vector $\vec{z}\in \text{Int}(C)$.

In general, if $\vec{x}=\sum_{i=1}^{n-1}\alpha_i\cdot\vec{w}_i  + \beta_{i_0}A(\vec{w}_{i_0})\in C$, then
\begin{equation}   
\label{igual1}
\langle\vec{v},\vec{x}\rangle=\sum_{i=1}^{n-1}\alpha_{i}\cdot\langle\vec{v},\vec{w}_i \rangle + \beta_{i_0}\cdot \langle \vec{v},A(\vec{w}_{i_0})\rangle=\beta_{i_0}\cdot \langle \vec{v},A(\vec{w}_{i_0})\rangle=\beta_{i_0}x_n\geq 0,
\end{equation}
\begin{equation}
\label{igual2}
\text{and }\ \langle (A^{-1})^t(\vec{v}),\vec{x}\rangle=\sum_{i=1}^{n-1}\alpha_i\cdot\langle (A^{-1})^t(\vec{v}),\vec{w}_i\rangle=\sum_{i=1}^{n-1}\alpha_i\cdot\langle \vec{v},(A^{-1})(\vec{w}_i)\rangle=-\sum_{i=1}^{n-1}\alpha_i|a_{i}|x_n\leq 0.    
\end{equation}

Given that $\vec{z}\in \text{Int}(C)$, then $\beta_{i_0}$ there must be a positive number. In fact, if $\beta_{i_0}=0$, then $\vec{x}\in H_v$ and, therefore,  $\langle\vec{v},\vec{x}\rangle=0$. On the other hand, there exists $\epsilon>0$ such that $\vec{x}-\epsilon\cdot\vec{v}\in C$. Finally, $\langle\vec{v},\vec{x}-\epsilon\cdot\vec{v}\rangle=-\epsilon\cdot\langle\vec{v},\vec{v}\rangle<0$. This inequality contradicts equation (\ref{igual1}).

Similarly, there must exist some index $i$ such that $\alpha_i> 0$ and $a_i\neq 0$. In fact, if $\alpha_i=0$ for all indices $i$ such that $a_i\neq 0$, then $\vec{x}\in A(H_v)=H_{(A^{-1})^tv}$. Moreover, there exists $\epsilon>0$ such that $\vec{x}+\epsilon\cdot(A^{-1})^t(\vec{v})\in C$. Finally, $\langle (A^{-1})^t(\vec{v}), \vec{x}+ \epsilon\cdot(A^{-1})^t(\vec{v})\rangle=\epsilon\cdot\langle (A^{-1})^t(\vec{v}), (A^{-1})^t(\vec{v})\rangle>0$. This inequality contradicts equation (\ref{igual2}).

Finally, we conclude that $\vec{z}$ satisfies $\langle\vec{v},\vec{z}\rangle>0$ and $\langle(A^{-1})^t(\vec{v}),\vec{z}\rangle<0.$ 
\end{proof}

\noindent\textit{Proof of Proposition \ref{matrixact}. } Of course, the sentence is trivial when $n=1$. In fact, $\mathbb{Z}$ admits only two different orders. Let us then suppose that the sentence holds for all $k\leq n-1$, for some $n\geq 1$. We only need to verify that the sentence holds for $n$ as well.

Let $P$ be a positive cone of $\mathbb{Z}^n$. We know that there exist a finite subnormal series
    $$\{1\}=C_{m+1}\vartriangleleft C_m \vartriangleleft\dots\vartriangleleft C_0=\mathbb{Z}^n$$
    and vectors $\vec{v}_0,\vec{v}_1,\dots,\vec{v}_t\in \mathbb{R}^n$ such that $\vec{x}\in P\cap(C_i\setminus C_{i+1})$ if and only if $\vec{x}\in C_i$ and $\langle\vec{v}_i,\vec{x}\rangle>0$.
    
Let us denote $\vec{v}_0=(x_1,\dots, x_n)$. Given this vector, we must consider two possible cases: when $x_n=0$ or $x_n\neq 0$.\\

\noindent\textit{i)} Suppose that $x_n=0$. 
If $A\in L_n$, then we have $(A^{-1})^t(\vec{v}_0)=\vec{v_0}$. Equivalently, $A(P\cap (\mathbb{Z}^n\setminus C_1))= P\cap (\mathbb{Z}^n\setminus C_1)$.
        From previous equality and equation (\ref{equation2}), it follows that 
        $$A(P)=P\cap (\mathbb{Z}^n\setminus C_1)\ \amalg\ A(P\cap C_1),$$
        for all $A\in L_n$.
        This new equation implies that we can restrict the action of $L_n$ to $C_1$. Of course, $C_1\simeq \mathbb{Z}^{k}$, for some $k<n$.

The reader can note that $A(\vec{e}_i) - \vec{e}_i\in \langle\ \vec{e}_n\ \rangle$ for all $i=1,\dots, n-1$, for all $A\in L_n$. On the other hand, $\vec{e}_n\in C_1$ because $\langle\vec{v},\vec{e}_n\rangle=0$. In fact, there must exist a free basis of $C_1$ that contains $\vec{e}_n$.
This implies that the action of $L_n$ restricted to $C_1$ admits a representation in the form $\varphi:L_n\rightarrow L_{k}$.

By the inductive hypothesis, we know that the action $L_k\curvearrowright\mathcal{LO}(\mathbb{Z}^{k})$ is smooth. This means that Orb$_{L_{k}}(P\cap C_1)$ is discrete in $\mathcal{LO}(\mathbb{Z}^{k})$. Due to Orb$_{L_n}(P\cap C_1)\subseteq \text{Orb}_{L_{k}}(P\cap C_1)$, we conclude that Orb$_{L_n}(P\cap C_1)$ is also discrete in $\mathcal{LO}(\mathbb{Z}^{k})$. Finally, by Lemma \ref{dico} we conclude that $P$ is not a condensed point.\\

\noindent\textit{ii)} Suppose that $x_n\neq0$.
We can assume that $x_n>0$ because $-I_n\in GL_{n}(\mathbb{Z})$.

Consider two matrices
$$A=\left ( 
        \begin{array}{cc}
            I_{n-1} & \vec{0}^t \\
            \vec{a} & 1        \end{array}
        \right )
\ \text{ and }\
B=\left ( 
        \begin{array}{cc}
            I_{n-1} & \vec{0}^t \\
            \vec{b} & 1        \end{array}
        \right )        
$$
where $\vec{a}=(a_1,\dots, a_{n-1}),$ $\vec{b}=(b_1,\dots,b_{n-1})$ and $b_i=\left \{ 
        \begin{array}{cc}
            0 & \text{if }a_i=0, \\
            \frac{a_i}{|a_i|} & \text{if } a_i\neq 0 ,
        \end{array}
        \right .$ for all indices $i$.

Let $i_0$ be the smallest index $i$ such that $a_i\neq 0$. From both matrices $A$ and $B$, we define vectors $\vec{w}_i\in \mathbb{R}^n$ as in the proof of Lemma \ref{lemacono}. Note that these vectors coincide for $A$ and $B$ because sgn$(a_i)=\text{sgn}(\frac{a_i}{|a_i|})$ for all $i$. Moreover, we consider $C$, the convex cone generated by $\vec{w}_1,\dots,\vec{w}_{i_0},A(\vec{w}_{i_0}).,\dots, \vec{w}_{n-1}, $ and $\tilde{C}$, the convex cone generated by $\vec{w}_1,\dots,\vec{w}_{i_0},$ $B(\vec{w}_{i_0}),$ $\dots, \vec{w}_{n-1}.$ The reader could also note that $\tilde{C}\subseteq C$.

Thanks to Lemma \ref{lemacono}, we know there exists a vector, which we will denote by $\vec{z}_b$, such that $\langle\vec{v}, \vec{z}_b\rangle>0$ but $\langle(B^{-1})^t(\vec{v}),\vec{z}_b\rangle<0$. Moreover, it is also true that $\langle(A^{-1})^t(\vec{v}),\vec{z}_b\rangle<0$. To verify this claim, it is sufficient to replicate the argument shown below (\ref{igual2}). 

Finally, we conclude that for all $A\in L_{n}\setminus\{I_n\}$ there exists a no null vector $\vec{b}\in \{-1,0,1\}^{n-1}$ such that $\langle\vec{v},\vec{z}_b\rangle>0$, but $\langle (A^{-1})^t(\vec{v}),\vec{z}_b\rangle<0$.
Equivalently, $\vec{z}_b\in P$, but $\vec{z}_b\not\in A(P)$. 

Therefore, we just need to consider the neighborhood 
$$U=\{\ Q\in \mathcal{LO}(\mathbb{Z}^n)\ |\ \vec{z}_b\in Q, \text{for every no null } \vec{b}\in \{-1,0,1\}^{n-1}\ \}$$ 
of $P$ in such a way that Orb$_{L_n}(P)\cap U=\{P\}$. 

We conclude that $P$ is not a condensed point. $\hfill\square$

\subsection{Proof of Theorem \ref{thmsmooth}}

The statement is obvious when $G$ is abelian, since all orders on an abelian left-orderable group are bi-orders and, therefore, a fixed points by the conjugacy action of $G$. Thus, we restrict ourselves to the case where $G$ is a nonabelian nilpotent group.

Suppose $[G,G]=\langle\ g\ \rangle\subseteq Z(G)$ for some element $g\in G$.
We will need the following:\\

        \noindent\textit{Claim. }There exists $g_0$ in $G$ such that $[G,G]\subseteq \langle\ g_0\ \rangle \subseteq Z(G)$ and $G/\langle \ g_0\ \rangle$  is a finitely generated torsion-free abelian group.\\ 

        \noindent\textit{Proof of Claim. } By derived subgroup's property, we know that $G/[G,G]$ is a finitely generated abelian group. This means that $G/[G,G]\simeq \mathbb{Z}^m\times T$ for some $m\geq1$ and a finite group $T$. 
        
        Let $a$ be an element in $G$ such that $a[G,G]\in T$. Then there exists a positive integer $k>1$ such that $a^k[G,G]=(a[G,G])^k=[G,G]$. Equivalently, $a^k\in [G,G]\subseteq Z(G)$. 

        We claim $a\in Z(G)$. In fact, if $h$ is an arbitrary element in $G$, we have $[h,a]\in [G,G]\subseteq Z(G)$. Given that $[h,a]\in Z(G)$, a simple calculus shows that $[h,a]^k=[h,a^k]=1$. Then $[h,a]=1$, because $G$ is torsion-free. Therefore, $a$ must be a central element of $G$. 

        We conclude that $\langle\ a,g\ \rangle \subseteq Z(G)$. Moreover, there exists an integer $t\neq 0$ such that $a^k=g^t$. This equality implies that $\langle\ a,g\ \rangle\simeq\mathbb{Z}$ and, therefore, $\langle \ a,g\ \rangle=\langle\ g_*\ \rangle$ for some $g_*\in Z(G)$. 
        Finally, by Isomorphism Theorems in Group Theory, we conclude that
        $$G/\langle\ g_*\ \rangle \simeq (G/[G,G])/(\langle\ g_*\ \rangle/[G,G])\simeq \mathbb{Z}^m\times T',$$
        where $T'$ is a proper quotient of $T$. 

        If we repeat this process finitely many times, then we will find an element $g_0$ in $G$ which satisfies what we want. $\hfill\square$
        \\ 

Let $\prec$ be a left-invariant order on $G$ and $P$ its associated positive cone. Due to $G$ being nilpotent, $\prec$ is a Conradian order. By Proposition \ref{series}, there exist a finite subnormal series
        \begin{equation}
        \label{serie}
        \{1\}=C_{l+1}\vartriangleleft C_l\vartriangleleft \dots \vartriangleleft C_0=G,    
        \end{equation}
and nondecreasing homomorphisms $f_i:C_i\rightarrow \mathbb{R}$ such that Ker$f_i=C_{i+1}$ for all $i=0,\dots,l$. 

If we consider the first homomorphism $f_0:C_0=G\rightarrow \mathbb{R}$, then it is obvious that $[G,G]\subseteq \text{Ker}f_0=C_1$. 
        Moreover, it must be true that $g_0\in C_1$, because $\mathbb{R}$ is a torsion-free group. 
        
        If $C_1$ is nonabelian, then $\{1\}\subsetneq [C_1,C_1]\subseteq \langle\ g_0\ \rangle$ and $f_1$ is not an injective homomorphism. We can repeat the previous argument for $f_1:C_1\rightarrow\mathbb{R}$ and, in this case, conclude that $[G,G]\subseteq \langle\ g_0\ \rangle\subseteq C_2$. 
        By an inductive process, we can repeat this argument until we find some index $i$ such that $C_i$ is an abelian group. Let us denote by $m\leq l$ this respective index. We conclude that
        $$ \{1\}\vartriangleleft C_m\vartriangleleft\dots \vartriangleleft C_1\vartriangleleft C_0=G,$$
        where $C_m$ is the largest abelian group in (\ref{serie}) that contains the element $g_0$.

        From the previous process, it follows that $[G,G]\subseteq C_i$ and, therefore, $C_i\vartriangleleft G$ for all $0\leq i\leq m$.
        Moreover, the reader can note that the conjugated action of $G$ on $C_i/C_{i+1}$ is trivial for all $0\leq 1\leq m-1$. Therefore, as in equation (\ref{equation2}), we have that
        $$g^{-1}Pg=P\cap (G\setminus C_m)\ \amalg\ g^{-1}(P\cap C_m)g$$
        for all $g\in G$. Previous equation and Lemma \ref{dico} imply that, if we want to verify that $P$ is not condensed, then we only need to verify that $P^*=P\cap C_m$ has a discrete conjugacy orbit. 
        To do this, we first determine the conjugated action of $G$ on $C_m$.

        Due to $G/\langle\ g_0\ \rangle$ being a torsion-free group, $C_m/\langle\ g_0\ \rangle$ is also a torsion-free group. Then there exist elements $a_1,a_2,\dots, a_n=g_0$ in $G$ such that $C_m=\langle\ a_1,a_2,\dots, a_n\ \rangle\ \simeq \mathbb{Z}^n$. More explicitly, the assignment $a_i\rightarrow \vec{e}_i=(0,\dots, 1,\dots,0)$ (with 1 in the $i$-th coordinate) holds and induces a group isomorphism. 

        We must remember that $[G,G]\subseteq \langle\ g_0\ \rangle$. For it, the conjugated action of $G$ restricted to $C_m$ admits a representation with the form $\varphi:G\rightarrow L_n$, such that $g^{-1}hg=\varphi(g)(h)$ for all $g\in G$ and all $h=a_1^{t_1}\dots a_n^{t_n}\in C_m$.

        By Proposition \ref{matrixact}, there exists a neighborhood $U$ of $P^*$ in $\mathcal{LO}(C_m)$ such that Orb$_{L_n}(P^*)\cap U=\{P^*\}$. In particular, Orb$_G(P^*)\cap U=\{P^*\}$ because Orb$_G(P^*)\subseteq \text{Orb}_{L_n}(P^*)$. By Proposition \ref{dico}, this implies that $P$ is not condensed in $\mathcal{LO}(G)$.  

        Finally, by Proposition \ref{condpoint}, the relation $E_{lo}(G)$ is smooth. This finishes our proof of Theorem \ref{thmsmooth}. $\hfill\square$

\begin{rem}
For last time, let us consider the $(2n+1)-$dimensional discrete Heisenberg group $\mathcal{H}_{2n+1}$. Recall that the elements
$$x_i:=\left ( 
\begin{array}{ccc}
1 & \vec{e_i} & 0\\
\vec{0}^t & I_n & \vec{0}^t\\
0 & \vec{0} & 1
\end{array}
\right ), \
y_i:=\left ( 
\begin{array}{ccc}
1 & \vec{0} & 0\\
\vec{0}^t & I_n & \vec{e_i}^t\\
0 & \vec{0} & 1
\end{array}
\right ) \text{ and }
z:=
\left ( 
\begin{array}{ccc}
1 & \vec{0} & 1\\
\vec{0}^t & I_n & \vec{0}^t\\
0 & \vec{0} & 1
\end{array}
\right )
$$
corresponds to a generating set of $\mathcal{H}_{2n+1}$ which satisfy the relations $[x_i,y_i]=z$ for all $i$, while any other commutator is trivial. Thus, it is easy to verify that $[\mathcal{H}_{2n+1},\mathcal{H}_{2n+1}]=Z(\mathcal{H}_{2n+1})=\langle\ z\ \rangle.$ Finally, the group $\mathcal{H}_{2n+1}$ satisfies the hypothesis of Theorem \ref{thmsmooth} and, therefore, $E_{lo}(\mathcal{H}_{2n+1})$ is smooth for all $n\geq 1.$
\end{rem}

\subsection{A nilpotent group with no smooth conjugacy action}

Let us denote by $\mathcal{N}$ the group defined by presentation 
$$\mathcal{N}=\langle\ b_1,b_2,a_1,a_2,a_3\ |\ [a_3,b_1]=a_1,\ [a_3,b_2]=a_2,\ \text{any other commutator is trivial}\ \rangle.$$

For this group, it is easy to check that $\mathcal{N}\simeq\langle\ b_1,b_2\ \rangle\prescript{}{_{A,B}}{\ltimes}\langle\ a_1,a_2,a_3\ \rangle,$ where $\langle\ b_1,b_2\ \rangle\simeq\mathbb{Z}^2$, $\langle\ a_1,a_2,a_3\ \rangle\simeq\mathbb{Z}^3$ and
$$A=\left ( 
\begin{array}{ccc}
    1 & 0 & 1 \\
    0 &  1 & 0 \\
    0 & 0 & 1
\end{array}
\right )
\ \text{ and }\ 
B=\left ( 
\begin{array}{ccc}
    1 & 0 & 0 \\
    0 &  1 & 1 \\
    0 & 0 & 1
\end{array}
\right ).$$

On the other hand, $[\mathcal{N},\mathcal{N}]=\langle\ a_1,a_2\ \rangle\subseteq Z(\mathcal{N})$. 

Note that the group $\mathcal{N}$ does not satisfy the hypothesis of Theorem \ref{thmsmooth}. In fact, the conjugacy orbit relation $E_{lo}(\mathcal{N})$ is not smooth. 
To prove this claim, we need to state a preliminary result.

\begin{lem}
\label{lemavector}
    Let $G$ be a group and $\varphi:G\rightarrow SL_{3}(\mathbb{Z})$ a group homomorphism. Let $\{\alpha, \beta,\gamma\}$ be  a $\mathbb{Z}$-linearly independent set and $P$ the positive cone of $\mathbb{Z}^3$ defined uniquely by $\vec{v}=(\alpha,\beta, \gamma)\in \mathbb{R}^3.$ If there exist $n,m\in \mathbb{Z}\setminus\{0\}$ and elements $g,h\in G$ such that 
    $$\varphi(g)=
\left ( \begin{array}{ccc}
1 & 0 & n\\
0 & 1 & 0\\
0 & 0 & 1
\end{array}
\right )
\text{ y }
\varphi(h)= 
\left ( \begin{array}{ccc}
1 & 0 & 0\\
0 & 1 & m\\
0 & 0 & 1
\end{array}
\right ),$$
then for all neighborhoods $U$ of $P$ there exists an element $k\in \langle\ g,h\ \rangle $ such that $\varphi(k)(P)\in U,$ but $\varphi(k)(P)\neq P$
\end{lem}

\begin{proof}
    Let us consider some neighborhood $U$ of $P$ in $\mathcal{LO}(\mathbb{Z}^3)$. This means that there exist elements $\vec{u}_1,\dots,\vec{u}_l\in \mathbb{Z}^3$ such that $U=\{Q\in \mathcal{LO}(\mathbb{Z}^3)\ :\ \vec{u}_1,\dots,\vec{u}_l\in Q\}$. In particular, $\langle\vec{v}, \vec{u}_i\rangle>0$ for all $i=1,\dots, l.$ 

    Due to $U$ being determined by finitely many elements, there must exist $\epsilon>0$ such that $\langle\vec{v} + \delta\vec{e},\vec{u}_i\rangle=\langle\vec{v}, \vec{u}_i\rangle+\delta\cdot\langle\vec{e}, \vec{u_i}\rangle>0$ for all $i=1,\dots,l$; for all $0\leq \delta<\epsilon.$ Here, $\vec{e}=(0,0,1)$. 
    On the other hand, there must exist no null integers $k_0$ and $t_0$ such that $0<-(k_0\alpha + t_0\beta)<\frac{\epsilon}{nm}$. This claim is true because $\langle\ \alpha,\beta\ \rangle$ is a dense subgroup of $\mathbb{R}$. 
    
    We then can consider $k=g^{k_0m}h^{t_0n}$ and the positive cone $\varphi(k)(P)$, which it is defined uniquely by the vector

    $$[\varphi(k)^{-1}]^t(v)=
    \left (
\begin{array}{ccc}
    1 & 0 & 0 \\
    0 & 1 & 0 \\
    -k_0nm & -t_0nm & 1  \\
\end{array}
\right )
\left (
\begin{array}{c}
     \alpha  \\
     \beta \\
     \gamma
\end{array}
\right ) = 
\vec{v} - nm(k_0\alpha + t_0\beta)\vec{e}.$$

A fast thought let us convince ourself that $\varphi(k)(P)\neq P,$ given that $n_0\neq 0$ and $m_0\neq 0$. Moreover, 
$\langle[\varphi(k)^{-1}]^t(\vec{v}), \vec{u}_i\rangle>0$ for all $i=1,\dots, l$. 

Finally, we conclude that $\varphi(k)(P)\in U\setminus\{P\}.$
This finishes the proof of Lemma 1.
\end{proof}

\begin{rem}
Note that if $\alpha, \beta$ and $\gamma$ are elements $\mathbb{Z}$-linearly independent, then there does not exist no null element in $\mathbb{Z}^3$ perpendicular to $\vec{v}=(\alpha,\beta,\gamma)$. Therefore, the positive cone $P$ in Lemma 1 is well defined and there does not exist a proper convex subgroup relative to $P$ in $\mathbb{Z}^3.$
\end{rem}

\begin{prop}
\label{newgroup}
    The relation $E_{io}(\mathcal{N})$ is not smooth.
\end{prop}
\begin{proof}
    Thanks to Proposition \ref{condpoint}, we only need to verify that the space $\mathcal{LO}(\mathcal{N})$ admits a condensed point. 
Consider then the subnormal series 
$$\{1\}\vartriangleleft\langle\ a_1,a_2,a_3\ \rangle \vartriangleleft \mathcal{N}.$$

Let us now provide an arbitrary order to $\mathcal{N}/\langle\ a_1,a_2,a_3\ \rangle\simeq \langle\ b_1,b_2\ \rangle$ and establish the assignation $a_1\rightarrow (1,0,0)$, $a_2\rightarrow (0,1,0)$, and $a_3\rightarrow (0,0,1)$. We provide the order $P^*$ to $\langle\ a_1,a_2,a_3\ \rangle$, where $P^*$ is uniquely defined by the vector $\vec{v}=(\sqrt{2},\sqrt{3},1) \in \mathbb{R}^3$. Finally, we denote by $P$ the lexicographic order on $\mathcal{N}$ defined by the previous subnormal series and the order on every factor. 

Note that conjugation action of $\mathcal{N}$ on $\mathcal{N}/\langle\ a_1,a_2,a_3\ \rangle$ is trivial. Inspired by equation (\ref{equation2}), we have that
\begin{equation}
\label{clever}
    g^{-1}Pg=P\cap (\mathcal{N}\setminus \langle\ a_1,a_2,a_3\ \rangle)\ \amalg\ g^{-1}P^*g,
\end{equation}
for all $g\in \mathcal{N}$. 

We will check that $P^*$ is a limit point of its conjugacy orbit with respect to $\mathcal{N}$. 

From group presentation of $\mathcal{N}$, conjugation action on $\langle\ a_1,a_2,a_3\ \rangle$ is given by representation $\varphi:\mathcal{N}\rightarrow SL_3(\mathbb{Z})$, where
$$\varphi(b_1)=
\left ( \begin{array}{ccc}
1 & 0 & 1\\
0 & 1 & 0\\
0 & 0 & 1
\end{array}
\right ),\hspace{0.5cm} 
\varphi(b_2)= 
\left ( \begin{array}{ccc}
1 & 0 & 0\\
0 & 1 & 1\\
0 & 0 & 1
\end{array}
\right )$$
and any other generator is mapped to identity matrix.

This situation satisfies the hypothesis of Lemma \ref{lemavector}. Therefore, for every neighborhood $U^*$ of $P^*$ there exists some element $g_0\in \langle\ b_1,b_2\ \rangle$ such that $\varphi(g_0)(P^*)\in U^*\setminus\{P^*\}$. Finally, we conclude that $P$ is a condensed point due to Lemma \ref{dico}.
\end{proof}

\subsection{The case of lower triangular matrices groups}

Of course, Proposition \ref{newgroup} gives us a proof of Theorem \ref{nosmooth}. However, we would like to finish this work presenting an infinitely family of nilpotent groups such that the conjugacy orbit equivalence relation on its space of orders is not smooth.

Let $k$ be a positive integer and consider the group of $k\times k$ lower triangular matrices over $\mathbb{Z}$, this means, 
\begin{equation}
\label{presentation}
N_k:=\left \{ \left ( 
\begin{array}{ccccc}
    1 & 0 & \dotsm & 0 & 0 \\
    a_{2,1} & 1 & \dotsm & 0 & 0\\ 
    \vdots & \vdots & \ddots & \vdots & \vdots\\
    a_{k-1,1} & a_{k_1,2} & \dotsm & 1 & 0\\
    a_{k,1} & a_{k,2} & \dotsm & a_{k,k-1} & 1
\end{array}\right )\ :\ a_{i,j}\in \mathbb{Z},\ \forall\ 1\leq j<i\leq k\  \right \}.
\end{equation}

From this definition, it follows immediately that $N_2$ and $N_3$ are groups isomorphic to $\mathbb{Z}$ and $\mathcal{H}_3$, respectively. Moreover, $N_1$ is nothing more than the trivial group. 

Let us denote by $e_{i,j}$ the matrix that it has 1 in its $(i,j)$-coordinate and $0$ in any other. Let us also define the matrices $E_{i,j}:=I_k+ e_{i,j}$ for all $1\leq j<i\leq k$. The reader can verify that $\{E_{i,j}\}_{1\leq j<i\leq k}$ corresponds to a generating set of $N_k$. Moreover, this set of generators satisfies $[E_{i,j},E_{j,l}]=E_{i,l}$ for $1\leq l<j<i\leq k$, while any other commutator is trivial. From this preferred generating system it follows that $N_k$ is a nilpotent group of class $k-1$.

In fact, if we denote by 
$$\{1\}=\gamma_k\leq \gamma_{k-1}\leq \dots \leq \gamma_1\leq \gamma_0=N_k$$
the lower central series of $N_k$, then $\gamma_l=\langle\ E_{i,j}\ :\ i-j>l \ \rangle$ for all $1\leq l<k$. Moreover, we know that $\gamma_i/\gamma_{i+1}\simeq \mathbb{Z}^{k-(i+1)} $ is a finitely generated abelian group for all $0\leq i \leq k-1$. 
We conclude that $N_k$ is a left-orderable group because it admits a subnormal series whose factors are left-orderable groups.

Of course, we would like to prove that $E_{lo}(N_K)$ is not smooth when $k\geq 4$. However, we will first study the case where $k=4$. For visual convenience, we write
$$N_4:=\left \{  \left (
        \begin{array}{cccc}
            1 & 0 & 0 & 0 \\
            e & 1 & 0 & 0 \\
            a & f & 1 & 0 \\
            c & b & d & 1 
        \end{array}
        \right )
        :  a,b,c,d,e,f \in \mathbb{Z} 
        \right \}.$$

and denote by $a,b,c,d,e$ and $f$ the generating matrices with $1$ at the corresponding position. We take this notation from Jorquera, Navas and Rivas' work \cite{Rivas}.

As before, the generating elements satisfy the relations $[d,f]=b$, $[d,a]=c$, $[f,e]=a$ and $[b,e]=c$, respectively. Meanwhile, every other commutator is trivial. The previous claim shows that $N_4\simeq \langle\ e,d\ \rangle \ltimes \langle\ a,b,c,f\ \rangle$, where $\langle\ e,d\ \rangle\simeq \mathbb{Z}^2$ and $\langle\ a,b,c,f\ \rangle\simeq \mathbb{Z}^4$.

\begin{rem}
Consider the subgroup $N=\langle\ a,b,c,d,f\ \rangle\leq N_4$.
The reader can verify that $N$ is normal in $N_4$, because $[f,e]=a$ and $[b,e]=c$.  Moreover, the only nontrivial commutators of generating elements in $N$ are $[d,f]=b$ and $[d,a]=c$. From these relations, it is not hard to see that $N$ is an isomorphic copy of $\mathcal{N}$ in $N_4$. We could conjecture then that we can lift, in some way, a condensed point in $\mathcal{LO}(N)$ to $\mathcal{LO}(N_4)$. Fortunately, Clay and Calderoni have proved a result that allows us to conclude what we want.
\end{rem}

\begin{prop} \cite[Corollary 3.4]{ClayCalderoni3}
{\it For a left-orderable group G, the following are equivalent:
\label{Convex}
\begin{itemize}
    \item[(1)] $E_{lo}(G)$ is smooth.
    \item[(2)] For every relative convex $C\leq G$ the conjugacy orbit equivalence relation $E_{lo}(C)$ is smooth.
\end{itemize} }    
\end{prop}

It is enough then to note that $N_4$ admits the subnormal series

$$\{1\}\vartriangleleft \langle\ b,c,d\ \rangle \vartriangleleft N\vartriangleleft N_4,$$

where $N_4/N \simeq \mathbb{Z}$. Therefore, $N$ is a convex subgroup of $N_4$ relative to every order determined by previous series.

By Proposition \ref{newgroup} and Proposition \ref{Convex}, we conclude the following:

\begin{coro}
\label{groupN4}
    The relation $E_{lo}(N_4)$ is not smooth.
\end{coro}   

Now, we would like to generalize the previous sentence to groups of higher-dimensional matrices.
Thanks to Corollary \ref{groupN4} and Proposition \ref{Convex}, it is enough to verify that $N_k$ is, in some way, a relative convex subgroup of $N_{k+1}$, for all $k\geq 4$. For it, we consider the canonical injective morphism $\phi:N_k\rightarrow N_{k+1}$, where 
$$B\longrightarrow \phi(B)=\left (
\begin{array}{cc}
    B & \vec{0}^t \\
    \vec{0} & 1
\end{array}
\right )$$
for all matrix $B$ in $N_k$.

Note that $\tilde{N_k}:=\phi(N_k)$ is not normal in $N_{k+1}$. However, and using the notation of generating elements for $N_{k+1}$ below (\ref{presentation}),  we can consider subgroup $A_{i}=\langle\ E_{k+1,1},\dots, E_{k+1,i}, \tilde{N_k}\ \rangle$ for every $1\leq i\leq k$.

The reader can check that 
$$\tilde{N_{k}}\vartriangleleft A_1\vartriangleleft \dots\vartriangleleft A_k=N_{k+1}.$$

Moreover, quotients $A_1/\tilde{N_k}$ and $A_{i+1}/A_i$ are infinite cyclic groups for all $1\leq i\leq k-1$. 

The previous sentence implies that $\tilde{N_k}$ is a convex subgroup of $N_{k+1}$ relative to every lexicographic order which is defined by previous subnormal series. 
Finally, we conclude with an inductive argument that relation $E_{lo}(N_k)$ is not smooth, for all $k\geq 4$.

\begin{coro}
    The relation $E_{lo}(N_k)$ is smooth if and only if $k=1,2,3$.
\end{coro}

Emir Molina Taucán

Departamento de Matemáticas, Universidad de Chile

Las Palmeras 3425, Ñuñoa, Santiago, Chile

Email: emir.molina@ug.uchile.cl

\end{document}